\documentclass[11pt]{article}
\usepackage{amssymb, amsmath, amsfonts, amsthm, graphics, mathrsfs, enumerate}
\usepackage{hyperref}
\usepackage{bm}
\hypersetup{
    colorlinks=true,
    linkcolor=blue,
    filecolor=magenta,      
    urlcolor=cyan,
}
\usepackage[hmargin=1 in, vmargin = 1 in]{geometry}
\usepackage{graphbox}
\usepackage{caption}
\usepackage{fancyvrb} % enable box around verbatim blocks
\usepackage{upquote} % ensure plain quotation marks in verbatim mode
\newtheorem{theorem}{Theorem}
\newtheorem{remark}{Remark}
\renewcommand{\d}{\mathrm{d}} % the differential "d"
\newcommand{\CC}{\mathbb{C}}
\newcommand{\RR}{\mathbb{R}}
\newcommand{\cK}{\mathcal{K}}
\newcommand{\xx}{\mathbf{x}}
\newcommand{\yy}{\mathbf{y}}
\newcommand{\uu}{\mathbf{u}}
\newcommand{\nn}{\mathbf{n}}
\newcommand{\rr}{\mathbf{r}}
\newcommand{\Brho}{\mbox{\boldmath$\rho$}}
\newcommand{\Btau}{\mbox{\boldmath$\tau$}}

\title{Zeta Correction: A New Approach to Constructing Corrected Trapezoidal Quadrature Rules for Singular Integral Operators}
\author{Bowei Wu, Per-Gunnar Martinsson}

\begin{document}

\maketitle

\begin{abstract}
A high order accurate quadrature rule for the discretization of boundary integral equations (BIEs) 
on closed smooth contours in the plane is introduced. 
This quadrature can be viewed as a hybrid of the spectral quadrature of Kress (1991) and the 
locally corrected trapezoidal quadrature of Kapur and Rokhlin (1997). 
The new technique combines the strengths of both methods, and attains high-order convergence, 
numerical stability, ease of implementation, and compatibility with the ``fast'' algorithms 
(such as the Fast Multipole Method or Fast Direct Solvers).
Important connections between the punctured trapezoidal rule and the Riemann zeta function are 
introduced, which enable a complete convergence analysis and lead to remarkably simple 
procedures for constructing the quadrature corrections.
The paper reports a detailed comparison between the new method and the methods of Kress, 
of Kapur and Rokhlin, and of Alpert (1999).
\end{abstract}

%sssssssssssssss
\section{Introduction}

This paper describes techniques for discretizing boundary integral equations (BIEs) of the form
\begin{equation}
\tau(\xx) + \int_{\Gamma}G(\xx,\yy)\,\tau(\yy)\,\d s(\yy) = f(\xx),\qquad\xx \in \Gamma,
\label{eq:BIEorig}
\end{equation}
where $\Gamma$ is a smooth closed contour in the plane and $\d s$ the arc length measure on $\Gamma$, where $f$ is a given smooth function, and where $G$ is a given kernel with a logarithmic singularity as $|\xx-\yy|\rightarrow 0$.
Equations such as \eqref{eq:BIEorig} commonly arise as reformulations of boundary value problems from potential theory, acoustic and electromagnetic wave propagation, fluid dynamics and many other standard problems in engineering and science.
When a PDE can be reformulated as an integral equation that is defined on the boundary of the domain, there are several advantages to doing so, in particular when the BIE is a Fredholm equation of the second kind.

A key challenge to using  \eqref{eq:BIEorig} for numerical work is that upon discretization, it leads to a system of linear equations with a dense coefficient matrix. Unless the problem is relatively small, it then becomes essential to deploy fast algorithms such as the Fast Multipole Method (FMM) \cite{rokhlin1987} or Fast Direct Solvers (FDS) \cite{2019_martinsson_book}. A second challenge is that the singularity in the kernel function $G$ means that if a standard quadrature rule is used when discretizing the integral, then only very slow convergence is attained as the number of degrees of freedom is increased.

This paper introduces a new family of quadrature rules for discretizing \eqref{eq:BIEorig} that are numerically stable even at high orders. For instance, a rule of order 42 is included in the numerical experiments. It is perfectly stable and capable of computing solutions to 14 correct digits with as few degrees of freedom as spectrally convergent quadratures such as the method of Kress \cite{kress1991boundary}. Moreover, unlike the Kress quadrature, it can easily be used in conjunction with fast solvers such as the FMM or FDS.

\subsection{Nystr\"om discretization and corrected trapezoidal rules}

Upon parameterizing the domain $\Gamma$ over an interval $[0,T]$, a BIE such as \eqref{eq:BIEorig} can be viewed as a one dimensional integral equation of the form
%eeeeeeeeeeeee
\begin{equation}
\tau(x) + \int_{0}^T \cK(x,y)\tau(y)\,\d y = f(x),\quad x\in[0,T].
\label{eq:BIE}
\end{equation}
In \eqref{eq:BIE}, the new kernel $\cK$ encodes both the parameterization and the original kernel $G$. Observe that all functions in \eqref{eq:BIE} are $T$-periodic, that $\tau$ and $f$ are smooth, and that $\cK$ is smooth except for a logarithmic singularity as $|x-y|\rightarrow 0$.

To discretize \eqref{eq:BIE} using the \emph{Nystr\"om method} \cite[\S12.2]{kress2014linear}, we consider the $N$-point \emph{periodic trapezoidal rule} (PTR)
\begin{equation}
\int_0^Tg(x)\,\d x \approx \sum_{n=1}^Ng(x_n)h
\label{eq:ptr}
\end{equation}
where $h=T/N$ and $x_n=nh$. When $g$ is smooth and periodic, the PTR converges super-algebraically as $N\to\infty$ \cite{trefethen2014exponentially}. The Nystr\"om method first collocates \eqref{eq:BIE} to the quadrature nodes $\{x_n\}$ of the PTR, and then approximates the integral by a quadrature supported on the same nodes with unknowns $\tau_n\approx\tau(x_n)$, yielding a linear system
%eeeeeeeeeeeee
\begin{equation}
\label{eq:nystromlinsys}
\tau_m + \sum_{n=1}^N \mathbf{K}(m,n)\,\tau_n = f(x_m),\quad  m= 1,\,\dots,\,N.
\end{equation}
The coefficient matrix $\mathbf{K}$ should be formed so that the approximation
\begin{equation}
\label{eq:nystromquad}
\sum_{n=1}^{N}\mathbf{K}(m,n)\,\tau(x_n)\approx
\int_{0}^T\cK(x_m\,y)\tau(y)\,\d y,
\qquad m = 1,\,\dots,\,N,
\end{equation}
holds to high accuracy.
If $\cK$ were to be smooth, this task would be easy, since we could then use the PTR \eqref{eq:ptr} without modifications and simply set
\begin{equation}
\label{eq:niceK}
\mathbf{K}(m,n) = \cK(x_m,x_n)h.
\end{equation}
In this unusual case, the solution $\{\tau_n\}$ of \eqref{eq:nystromlinsys} will converge super-algebraically to $\{\tau(x_n)\}$ as $N\to\infty$.

In the more typical case where $\cK(x,y)$ is logarithmically singular at $y=x$, some additional work is required to attain high order convergence in \eqref{eq:nystromquad}. Let us start by describing two existing methods that resolve this problem --- the Kress quadrature and the Kapur-Rokhlin quadrature --- that are closely related to the new quadrature that we will describe.

Kress \cite{kress1991boundary} introduced a quadrature that is spectrally accurate for any periodic function $g$ of the form
\begin{equation}
g(x) = \varphi(x)\log\left(4\sin^2\frac{\pi x}{T}\right)+\psi(x)
\label{eq:kress_split}
\end{equation}
where $\varphi$ and $\psi$ are smooth functions with known formulae. Kress integrates the first term of \eqref{eq:kress_split} by Fourier analysis and the second term using the PTR, resulting in a corrected trapezoidal quadrature where all the PTR weights are modified. It is not obvious how a BIE scheme based on the Kress quadrature can be accelerated by the existing fast algorithms. We will use a ``localized'' analog of the analytic split \eqref{eq:kress_split} to develop our quadrature.

Kapur and Rokhlin \cite{kapur1997high}, on the other hand, constructed a family of quadratures for a variety of singular functions by correcting the trapezoidal rule locally near the singularity. These quadratures have correction weights that are essentially independent of the grid spacing $h$ and possess an essential benefit in that nearly all entries of the coefficient matrix are given by the simple formula \eqref{eq:niceK}; only a small number of entries near the diagonal are modified. This local nature of Kapur-Rokhlin quadrature makes it very easy to combine it with the FMM and other fast algorithms. For the logarithmic singularity, two different quadratures are developed. The first quadrature is for functions of the ``nonseparable'' form
\begin{equation}
g(x) = \varphi(x)\log|x|+\psi(x)
\label{eq:nonseparable}
\end{equation}
where the formulae of the smooth functions $\varphi$ and $\psi$ can be \emph{unknown}. This ``nonseparable'' quadrature ignores the data at $x=0$ completely and modifies a few trapezoidal weights on both sides of the singular point. The magnitude of the correction weights (tabulated in \cite[Table 6]{kapur1997high}) grow rapidly with the order of the correction, and moreover, some of the weights are negative. These properties mean that Kapur Rokhlin quadrature becomes less useful at higher orders (say order higher than 6), since the resulting coefficient matrix can be far worse conditioned than the underlying BIE \cite[Sec.~7.3]{hao2014high}. The second Kapur-Rokhlin quadrature is for functions of the ``separable'' form
\begin{equation}
g(x) = \varphi(x)\log|x|.
\label{eq:separable}
\end{equation}
Unlike the first rule, this ``separable'' quadrature also uses the data $\varphi(0)$ at the singular point for its correction; the correction weights (tabulated in \cite[Table 7]{kapur1997high}) are uniformly bounded regardless of the correction order, and decay rapidly away from the singular point.
Despite the excellent stability properties of the second Kapur-Rokhlin rule, it has received little attention due to the simple fact that the kernels arising from BIEs typically are of the nonseparable type \eqref{eq:nonseparable}. (In fact, the first rule is often referred to as ``the Kapur-Rokhlin'' rule.)

To avoid the issue of large correction weights, Alpert \cite{alpert1999hybrid} developed a hybrid Gauss-trapezoidal quadrature that uses an optimized set of correction points that are off the uniform trapezoidal grid, whose weights are very well-behaved. For additional details on high order accurate techniques for discretizing \eqref{eq:BIE}, as well as a discussion of their relative advantages, we refer to the survey \cite{hao2014high}.

\subsection{Contributions}

This paper describes a quadrature rule that is closely related to the neglected second Kapur-Rokhlin rule for separable functions. The new rule works almost exactly the same in practice in that it involves a small correction stencil that includes a correction weight at the origin. Both rules display excellent numerical stability and lead to discretized systems that are as well conditioned as the original equation. The key innovation is that the new rule that we present is applicable to functions of the form
\begin{equation}
g(x) = \varphi(x)\log|\rr(x)|+\psi(x)
\label{eq:loc_geom_split}
\end{equation}
where $\rr(x)$ is a smooth parametric curve in $\mathbb{R}^n$, and the smooth function $\varphi$ is known. Therefore, we have generalized the second Kapur-Rokhlin rule (which only works for the integrand \eqref{eq:separable} on $\mathbb{R}$) to new rules that can handle logarithmic kernels on curves in higher dimensions, making them applicable to a wide range of BIEs.

Our local kernel split \eqref{eq:loc_geom_split} is analogous to Kress' analytic split \eqref{eq:kress_split}; the main difference is that Kress has split down to the parameter level (with periodization) for the logarithmic component, while we only split locally down to the geometry level. We analyze the error of applying the punctured trapezoidal rule to the singular component of \eqref{eq:loc_geom_split} based on the lattice sum theory (see, e.g.~\cite{borwein2013lattice,wu2020corrected}). This results in an error expansion with coefficients that are explicitly computable using the Riemann zeta function, a fact we called the ``zeta connections.'' From these error coefficients, we construct local correction weights in the fashion of Kapur and Rokhlin. As it turns out, the correction weights constructed this way have a component that is the weights of the ``separable'' Kapur-Rokhlin quadrature mentioned above, while the remaining component depends on the explicit kernel split \eqref{eq:loc_geom_split}. (Remarkably, the zeta connection simplifies the construction of the separable Kapur-Rokhlin correction weights to the extent that Table 7 of \cite{kapur1997high} can be computed with three lines of code, as shown in Figure \ref{fig:code}.)

The zeta connection associated with the singular function $|x|^{-z}, -1<z<1$, was first introduced by Marin et.~al.~\cite{marin2014corrected}. In this paper, we extend this connection to a ``differential zeta connection'' (Theorem \ref{thm:zeta_prime_connection}) associated with the logarithmic singularity. We would also like to point out that the zeta connection has recently been generalized to higher dimensions by the authors in \cite{wu2020corrected}. Combining the ``differential zeta connection'' in this paper with the higher dimensional zeta connection of \cite{wu2020corrected}, we expect that a rigorous theory can be developed for higher-dimensional logarithmic quadratures such as the one developed by Aguilar and Chen \cite{aguilar2002high}. On the other hand, there is also the connection between the zeta function and the endpoint corrections of the trapezoidal rule for regular functions, which has been established much earlier, see \cite{navot1961extension,alpert1999hybrid}. 

We should mention that the zeta connections presented in this paper can be alternatively derived from the extended Euler-Maclaurin formulae of Navot \cite{navot1961extension,navot1962further}. In particular, the result of \cite{navot1962further} had been combined with Richardson extrapolation to construct high-order quadrature rules by Sidi and Israeli, see \cite{sidi1988quadrature}.

The technique presented in this manuscript for the case of contour integrals in two dimensions can be extended to higher dimensions as well, as is demonstrated in \cite{wu2020corrected} for the case of singular integral operators on surfaces in $\mathbb{R}^{3}$. The role played by the Riemann zeta function in the present paper must in higher dimensions instead be carried by the more general Epstein zeta function \cite{epstein1903theorie}.

\subsection{Organization}

In Section \ref{sc:zeta_connect}, we introduce the theory for the local correction of the trapezoidal quadrature and its connection to the zeta function. We extend this connection to construct a quadrature for the logarithmic singularity, which recovers the ``separable'' Kapur-Rokhlin quadrature but with much simpler computations. In Section \ref{sc:log_on_curve}, we generalize this Kapur-Rokhlin rule to logarithmic kernels on planar curves using a localized version of Kress' kernel split, developing quadratures for the Laplace and Helmholtz layer potentials. Finally in Section \ref{sc:numerical}, we present numerical examples of solving BIEs associated with the Helmholtz and Stokes equations and compare our quadrature method with existing state-of-the-art methods.

%sssssssssssssss
\section{Corrected trapezoidal rules and the zeta connections}
\label{sc:zeta_connect}

In this section, we introduce the theory for the local correction of the trapezoidal rule for singular functions. In particular, we introduce the ``zeta connections,'' based on which simple and powerful procedures are presented (see Remark \ref{rmk:complex_step}) for constructing quadratures for functions with algebraic or logarithmic branch-point singularities.

\subsection{Singularity correction by moment fitting}

To set up the notation for our discussion, we let $I[g]$ denote the integral of a function $g(x)$ on the interval $[-a,a]$, where $a>0$, and where $g$ may be singular at $x=0$. We denote the \emph{punctured trapezoidal rule} discretization of $I[g]$ as
\begin{equation}
T^0_h[g] = \sideset{}'\sum_{n=-M}^M g(x_n)h.
\label{eq:punctured}
\end{equation}
for some integer $M>0$, where $h=a/(M+\frac{1}{2})$ and $x_n=nh$, and where the prime on the summation sign indicates that $n=0$ is omitted. (Note that by our definition of $h$, the endpoints are not included as quadrature nodes, thus the usual $1/2$ factor for the weights at the endpoints is not needed. However, this choice is just for convenience and using the usual trapezoidal rule does not affect our analysis.) 

We are interested in singular integrals of the form
%eeeeeeeeeeeee
\begin{equation}
I[s\cdot \tau] = \int_{-a}^a s(x)\tau(x)\,\d x
\end{equation}
where $s(x)$ has an isolated, integrable singularity at $x=0$ and $\tau(x)$ is smooth. We assume that either the integrand $s(x)\tau(x)$ is periodic or that $\tau(x)$ is compactly supported in $[-a,a]$ so that the only obstruction to high-order convergence is the singularity of $s(x)$ at $x=0$. In general, boundary corrections to the trapezoidal rule can be introduced near $x=\pm a$ independent of the singularity at $0$; see \cite{alpert1995high}, for example, for more detail. 

To analyze the error of the approximation $I[s\cdot \tau]\approx T^0_h[s\cdot \tau]$, we decompose the integrand into a \emph{regular component} and a \emph{local component}:
\begin{equation}
s\cdot\tau = s\cdot \tau\cdot (1-\eta) + s\cdot \tau\cdot\eta
\end{equation}
where $\eta\in C^p_c([-a,a])$ is a smooth and compactly supported function that is at least $p$-time continuously differentiable and which satisfies $\eta(0)=1$ and $\eta(x) \equiv \eta(-x)$. For the regular component $s\cdot \tau\cdot (1-\eta)$, the trapezoidal discretization
\begin{equation}
I[s\cdot \tau\cdot(1-\eta)] = T^0_h[s\cdot \tau\cdot(1-\eta)] + O(h^p)
\label{eq:smooth_trapz}
\end{equation}
holds for any $p>0$ (note that the integrand is $0$ at $x=0$). Thus the overall convergence of $T^0_h[s\cdot \tau]$ is restricted by the error in the local component $s\cdot \tau\cdot\eta$. 

Using the idea of moment fitting \cite{keast1979structure,kapur1997high}, this local error can be corrected by fitting a set of moment equations for monomials $\tau^{(k)}(x)=x^k$ up to a sufficiently high degree as follows
%eeeeeeeeeeeee
\begin{equation}
h^\alpha\sideset{}{''}\sum_{j =-K}^K w_j^h\,(jh)^k\,\eta(jh) = I[s\cdot \tau^{(k)}\cdot\eta] - T^0_h[s\cdot \tau^{(k)}\cdot\eta],\qquad 0\leq k\leq 2K
\label{eq:moment_eq}
\end{equation}
where $K\geq0$ is an integer and the factor $h^\alpha$ depends explicitly on the singularity at $x=0$, and where the double prime on the summation sign indicates that the $j=0$ term is multiplied by 2. (For convenience we let $0^0:=1$.) There are $2K+1$ equations in \eqref{eq:moment_eq} for the $2K+1$ unknowns $\{w_j^h\}$. Then combining (\ref{eq:smooth_trapz}--\ref{eq:moment_eq}) yields a locally corrected trapezoidal quadrature
\begin{equation}
I[s\cdot \tau] \approx T^0_h[s\cdot \tau] + h^\alpha\sideset{}{''}\sum_{j =-K}^K w_j^h\tau(jh)
\label{eq:corrected_trapz}
\end{equation}
which is high-order accurate as $h\to0^+$.

%sssssssssssssss
\subsection{The $|x|^{-z}$ singularity and converged correction weights}

The fact that the weights $\{w_j^h\}$ in \eqref{eq:corrected_trapz} depend on $h$ is inconvenient in practice. Fortunately, this can be remedied by what we called the ``zeta connection.'' When $s(x) = |x|^{-z}$, $-1<z<1$, Marin et.\ al.\ \cite{marin2014corrected} showed that by letting $h\to0$ only in the $\eta(jh)$ term in the moment equations \eqref{eq:moment_eq}, the right-hand side of \eqref{eq:moment_eq} in this limit can be represented as Riemann zeta function values, and the corresponding limiting weights $\{w_j\}$ are independent of $h$; more importantly, they proved that using these converged weights in the quadrature \eqref{eq:smooth_trapz} in place of $\{w_j^h\}$ does \emph{not} affect the order of accuracy. We summarize this result of \cite{marin2014corrected} in this section, which will serve as the foundation for our extensions to other quadratures.

We first introduce the important concept of ``converged correction weights.'' 

Substitute $s(x)=|x|^{-z}$ in \eqref{eq:moment_eq} and let $\alpha=1-z$, we have
\begin{equation}
\begin{aligned}
h^{1-z}\sideset{}{''}\sum_{j =-K}^K w_j^h\,(jh)^k\,\eta(jh) &= \int_{-a}^a |x|^{-z}x^k\eta(x)\,\d x - \sideset{}'\sum_{n=-M}^M |nh|^{-z}(nh)^k\eta(nh)\,h\\
&= h^{1-z+k}\left(\int_{-N-\frac{1}{2}}^{N+\frac{1}{2}} |x|^{-z}x^k\eta(xh)\,\d x - \sideset{}'\sum_{n=-M}^M |n|^{-z}n^k\eta(nh)\right)\\
&= h^{1-z+k}\left(\int_{-\infty}^{\infty} |x|^{-z}x^k\eta(xh)\,\d x - \sideset{}'\sum_{n=-\infty}^\infty|n|^{-z}n^k\eta(nh)\right)
\end{aligned}
\label{eq:derivation1}
\end{equation}
where the last equality holds because $\eta$ is compactly supported. Note that both $|x|^{-z}$ and $\eta(x)$ are even, so by requiring that 
$$w_j^h \equiv w_{-j}^h,\quad j = 0,\dots,K$$
both sides of \eqref{eq:derivation1} vanish for all $k$ odd. Using this symmetry and eliminating $h^{1-z+k}$ on both sides yield
%eeeeeeeeeeeee
\begin{equation}
\sum_{j =0}^K w_j^h\,j^{2k}\,\eta(jh) = \int_{0}^{\infty} |x|^{-z+2k}\eta(xh)\,\d x - \sum_{n=1}^\infty|n|^{-z+2k}\eta(nh),\qquad k = 0,\dots, K
\label{eq:moment_eq_sym}
\end{equation}
From here we define the \emph{converged correction weights} $w_j, j=0,\dots,K,$ to be
\begin{equation}
w_j:=\lim_{h\to0^+}w_j^h, 
\label{eq:converged_wei_def}
\end{equation}
such that $w_j^h$ are the solution of \eqref{eq:moment_eq_sym}. To further simplify the equations, we need the following theorem.
%tttttttttthhhhhhhhhhh
\begin{theorem} [Zeta connection]
\label{thm:zeta_connection}
For all $z\in\CC\setminus\{1\}$,
\begin{equation}
\lim_{h\to0^+}\left(\sum_{n=1}^\infty|n|^{-z}\eta(nh) - \int_{0}^{\infty} |x|^{-z}\eta(xh)\,\d x\right) = \zeta(z).
\label{eq:zeta_connection}
\end{equation}
Consequently, the converged weights $w_j$, as defined by \eqref{eq:converged_wei_def}, are the solution of the system
\begin{equation}
\sum_{j =0}^K w_j\,j^{2k} = -\zeta(z-2k),\qquad k=0,1,\dots,K
\label{eq:moment_eq_zeta}
\end{equation}
\end{theorem}
\begin{proof}
The zeta connection \eqref{eq:zeta_connection} is proved in \cite[Lemma A2]{marin2014corrected}. Then taking $h\to0^+$ in \eqref{eq:moment_eq_sym} yields \eqref{eq:moment_eq_zeta}. Note that we used the fact that $\eta(0)=1$. (Alternatively, \eqref{eq:zeta_connection} can also be derived from the extended Euler-Maclaurin formula of \cite{navot1961extension}.)
\end{proof}
Based on the zeta connection, the next theorem constructs a high-order corrected trapezoidal rule using converged correction weights.
%tttttttttthhhhhhhhhhh
\begin{theorem}
\label{thm:quad_|x|^-z}
For $s(x) = |x|^{-z}$, one has the locally corrected trapezoidal rule
\begin{equation}
I[s\cdot\tau\cdot\eta] = T^0_h[s\cdot\tau\cdot\eta] + h^{1-z}\sum_{j = 0}^{K}w_j\big(\tau(jh)+\tau(-jh)\big) + O(h^{2K+3-z}),
\label{eq:quad_|x|^-z}
\end{equation}
where the correction weights $\{w_j\}$ are the solution of \eqref{eq:moment_eq_zeta}, and where $\eta\in C^p_c([-a,a])$, $p>2K+1,$ must have at least $2K+1$ vanishing derivatives at $0$, i.e.\
\begin{equation}
\eta(0) = 1,\quad \eta^{(k)}(0)=0,\quad k=1,2,\dots,2K+1.
\end{equation}
\end{theorem}
\begin{proof}
See Theorem 3.7 and Lemma 3.8 of \cite{marin2014corrected}.
\end{proof}

%sssssssssssssssssss
\subsection{Logarithmic singularity and the differential zeta connection}

One can also construct a quadrature for the logarithmic singularity by extending the zeta connection \eqref{eq:zeta_connection} based on the following simple observation:
$$
\frac{\d}{\d z}|x|^{-z} = -|x|^{-z}\log|x|.
$$
This leads to the next two theorems that are completely analogous to Theorems \ref{thm:zeta_connection} and \ref{thm:quad_|x|^-z}.
%tttttttttthhhhhhhhhhh
\begin{theorem} [Differential zeta connection]
\label{thm:zeta_prime_connection}
For all $z\in\CC\setminus\{1\}$,
\begin{equation}
\lim_{h\to0^+}\left(\sum_{n=1}^\infty-|n|^{-z}\log|n|\,\eta(nh) - \int_{0}^{\infty} -|x|^{-z}\log|x|\eta(xh)\,\d x\right) = \zeta'(z).
\label{eq:zeta_prime_connection}
\end{equation}
Consequently, if we define the converged weights $w_j:=\lim_{h\to0^+}w_j^h, j=0,\ldots,K,$ such that $w_j^h$ are the solution of the system
%eeeeeeeeeeeee
\begin{equation}
\sum_{j =0}^K w_j^h\,j^{2k}\,\eta(jh) = \int_{0}^{\infty} -|x|^{2k}\log|x|\,\eta(xh)\,\d x - \sum_{n=1}^\infty-|n|^{2k}\log|n|\,\eta(nh),\quad k = 0,\dots, K
\label{eq:moment_eq_sym_log}
\end{equation}
then $\{w_j\}$ are the solution of the system
\begin{equation}
\sum_{j =0}^K w_j\,j^{2k} = -\zeta'(-2k),\qquad k=0,1,\dots,K
\label{eq:moment_eq_zeta_prime}
\end{equation}
\end{theorem}
\begin{proof}
Taking the derivative with respect to $z$ under the limit sign on both sides of \eqref{eq:zeta_connection} yields \eqref{eq:zeta_prime_connection}, which is justified since the expression under the limit sign is analytic in $z$. Then taking $h\to0^+$ in \eqref{eq:moment_eq_sym_log} yields \eqref{eq:moment_eq_zeta_prime}. (An alternative proof can be derived from the second extended Euler-Maclaurin formula of Navot \cite{navot1962further}.)
\end{proof}
%tttttttttthhhhhhhhhhh
\begin{theorem}
\label{thm:quad_log|x|}
For $s(x) = -\log|x|$, one has the locally corrected trapezoidal rule
\begin{equation}
I[s\cdot\tau\cdot\eta] = T^0_h[s\cdot\tau\cdot\eta] -\tau(0)h\log h+h\sum_{j = 0}^{K}w_j\big(\tau(jh)+\tau(-jh)\big) + O(h^{2K+2}),
\label{eq:quad_log|x|}
\end{equation}
where the correction weights $\{w_j\}$ are the solution of \eqref{eq:moment_eq_zeta_prime}, and where $\eta\in C^p_c([-a,a])$ satisfies the same conditions as in Theorem \ref{thm:quad_|x|^-z}.
\end{theorem}
\begin{proof}
The local error of the punctured trapezoidal rule for $s(x)=-\log|x|$ is
\begin{equation}
\begin{aligned}
I[s\cdot\tau\cdot\eta] - T^0_h[s\cdot\tau\cdot\eta] &= \int_{-a}^a -\log|x|\,\tau(x)\eta(x)\,\d x - \sideset{}'\sum_{n=-M}^M -\log|nh|\,\tau(nh)\eta(nh)\,h\\
&= \left(\int_{-a}^a -\log\left|\frac{x}{h}\right|\,\tau(x)\eta(x)\,\d x - \sideset{}'\sum_{n=-M}^M -\log|n|\,\tau(nh)\eta(nh)\,h\right)\\
&\quad -\log h\, \left(\int_{-a}^a \tau(x)\eta(x)\,\d x - \sideset{}'\sum_{n=-M}^M \tau(nh)\eta(nh)\,h\right)
\end{aligned}
\label{eq:derivation2}
\end{equation}
Notice that in the second parentheses, the integral is smooth so the regular trapezoidal rule converges super-algebraically, i.e.\ using the fact that $\eta(0)=1$ we have
\begin{equation}
\log h\, \left(\int_{-a}^a \tau(x)\eta(x)\,\d x - \sideset{}'\sum_{n=-M}^M \tau(nh)\eta(nh)\,h\right) = (h\log h)\,\tau(0)+O(h^p) \text{ as $h\to0^+$}
\label{eq:derivation2.2}
\end{equation}
holds for any $p>0$. On the other hand, using the idea of moment fitting and following a similar derivation as (\ref{eq:derivation1}--\ref{eq:moment_eq_sym}), the terms in the first parentheses of \eqref{eq:derivation2} are approximated by
\begin{equation}
h\sum_{j = 0}^{K}w_j^h\big(\tau(jh)+\tau(-jh)\big)+O(h^{2K+3})
\label{eq:derivation2.3}
\end{equation}
where $w_j^h$ are the solution of \eqref{eq:moment_eq_sym_log}. Then substituting (\ref{eq:derivation2.2}--\ref{eq:derivation2.3}) into \eqref{eq:derivation2} gives
$$
I[s\cdot\tau\cdot\eta] = T^0_h[s\cdot\tau\cdot\eta] -\tau(0)h\log h+h\sum_{j = 0}^{K}w_j^h\big(\tau(jh)+\tau(-jh)\big)+O(h^{2K+3}).
$$
The above equation implies \eqref{eq:quad_log|x|} once it is shown that
$$
|w_j - w_j^h| = O(h^{2K+1}),\quad j = 0,1,\dots,K,
$$
which in turn can be proved by showing that the limit \eqref{eq:zeta_prime_connection} converges as $O(h^{2K+1})$ uniformly for $0\leq z\leq 2K$ given the condition of $\eta$. This last statement can be proved following almost verbatim the proofs of Theorem 3.1 and Lemma 3.3 in \cite{marin2014corrected} by replacing $|x|^{-z}$ with $|x|^{-z}\log|x|$ therein, hence we omit the detail here.
\end{proof}

The logarithmic quadrature \eqref{eq:quad_log|x|} is equivalent to the ``separable'' Kapur-Rokhlin quadrature developed in \cite[\S 4.5.1]{kapur1997high}, hence the correction weights are identical (up to a minus sign) to those given in \cite[Table 7]{kapur1997high}; however, the differential zeta connection has greatly simplified the construction of these weights.

\begin{remark}
\label{rmk:complex_step}
In practice, the value $\zeta'(z)$, $z\in\RR$, can be approximated using ``complex step differentiation'' \cite{squire1998using} as
\begin{equation}
\zeta'(z)=\frac{\mathrm{Im}\,\zeta(z+i\delta)}{\delta} + O(\delta^2)
\end{equation}
where $i=\sqrt{-1}$, $0<\delta\ll1$. This formula is free of cancellation errors that plagued typical finite difference methods. For instance, using $\delta\approx10^{-9}$ will yield an approximation of full double-precision accuracy.

On the other hand, as mentioned in \cite{kapur1997high}, the Vandermonde system \eqref{eq:moment_eq_zeta_prime} is ill-conditioned for large $K$. Thus when precomputing the weights $\{w_j\}$, \eqref{eq:moment_eq_zeta_prime} should be solved symbolically or under extended precision. Simple code snippets that compute $\{w_j\}$ for any given $K$ are given in Figure \ref{fig:code}.
\end{remark}

\begin{figure}[htbp]
\raggedright
MATLAB:
\begin{Verbatim}[frame=single]
rhs = -imag(zeta(-vpa(0:2:2*K)'+1i*eps)/eps);
V = vpa(0:K).^((0:2:2*K)');
w = double(V\rhs);
\end{Verbatim}
Julia:
\begin{Verbatim}[frame=single]
using SpecialFunctions
rhs = -imag.(zeta.(-(0:2:2*K).+im*eps())./eps())
V = (BigFloat.(0:K))'.^(0:2:2*K)
w = V\rhs
\end{Verbatim}
Mathematica:
\begin{Verbatim}[frame=single]
Unprotect[Power]; Power[0,0]=1; Protect[Power]; (* set 0^0=1 *)
rhs = -Im[Zeta[-2(Range[K+1]-1)+I*$MachineEpsilon]/$MachineEpsilon];
V = Array[(#2-1)^(2(#1-1))&, {K+1,K+1}];
w = Inverse[V].rhs;
\end{Verbatim}
\caption{Code that given $K$ (such that the correction order is $2K+2$) constructs the correction weights for integration against $-\log|z|$, as described in Remark \ref{rmk:complex_step} (where $\delta$ is chosen to be the machine precision \texttt{eps} for simplicity). The code generates \cite[Table 7]{kapur1997high} up to a minus sign.}
\label{fig:code}
\end{figure}

%sssssssssssssssssss
\section{Logarithmic kernels on curves}
\label{sc:log_on_curve}

In this section, we extend the ``separable'' Kapur-Rokhlin rule \eqref{eq:quad_log|x|} to construct our ``zeta-corrected quadrature.'' We will combine the differential zeta connection with local kernel splits (that are analogous to Kress' global analytic split \eqref{eq:kress_split}) to construct quadratures for some important logarithmic kernels on closed curves, including the Laplace and Helmholtz layer potentials; the quadrature for the Laplace single-layer potential will also be applied to integrate the Stokes potential in Section \ref{sc:numerical}.

\subsection{Laplace kernel}
Consider a smooth closed curve $\Gamma$ parameterized by a $2a$-periodic function $\Brho(x)\in\RR^n$. We consider the Laplace single-layer potential (SLP) from $\Gamma$ to $\Brho(0)\in\Gamma$ (for the general case, one simply replace $\Brho(0)$ with any other target point on $\Gamma$),
\begin{equation}
S[\tau](\Brho(0)) := \int_{-a}^a \big(-\log r(x)\big)\tau(x)\,|\Brho'(x)|\,\d x
\label{eq:laplace_ker_curve}
\end{equation}
where $r(x) := |\Brho(0)-\Brho(x)|$ and $\tau(x)\equiv \tau(\Brho(x))$. The next theorem extends Theorem \ref{thm:quad_log|x|} to construct a corrected quadrature for \eqref{eq:laplace_ker_curve}.

%tttttttttthhhhhhhhhhh
\begin{theorem}
\label{thm:quad_on_curve_log|x|}
For the Laplace SLP \eqref{eq:laplace_ker_curve}, one has the locally corrected trapezoidal rule
\begin{equation}
S[\tau](\Brho(0)) = T^0_h[s\cdot\tilde{\tau}] -\tilde{\tau}(0)h\log(|\Brho'(0)|\,h)+h\sum_{j = 0}^{K}w_j\big(\tilde{\tau}(jh)+\tilde{\tau}(-jh)\big) + O(h^{2K+2}),
\label{eq:quad_on_curve_log|x|}
\end{equation}
where $s(x) = -\log r(x)$ and $\tilde{\tau}(x):=\tau(x)|\Brho'(x)|$ is smooth, and where the correction weights $w_j$ are exactly the same as in Theorem \ref{thm:quad_log|x|}. 
\end{theorem}
Note that the only difference of \eqref{eq:quad_on_curve_log|x|} from \eqref{eq:quad_log|x|} is that $\log h$ is replaced with $\log(|\Brho'(0)|h)$ and $\tau$ is replaced with $\tilde{\tau}$ .
\begin{proof}
First we analyze the singularity of $\log r(x)$. Note that
\begin{equation}
\begin{aligned}
\log r(x) &= \frac{1}{2}\log\Big(|\Brho'(0)x|^2 + (r^2(x)-|\Brho'(0)x|^2)\Big) \\
&= \log(|\Brho'(0)x|) +\frac{1}{2}\log\left(1 + \frac{r^2(x)-|\Brho'(0)x|^2}{|\Brho'(0)x|^2}\right)
\end{aligned}
\label{eq:derivation3}
\end{equation}
which can be rewritten as
\begin{equation}
\frac{1}{2}\log\left(1 + \frac{r^2(x)-|\Brho'(0)x|^2}{|\Brho'(0)x|^2}\right) = \log \frac{r(x)}{|\Brho'(0)x|}
\label{eq:derivation3p}
\end{equation}
We will show that \eqref{eq:derivation3p} is smooth, thus the only singular term in \eqref{eq:derivation3} is $\log(|\Brho'(0)x|)$.
To this end, first expand $\Brho(x)$ as a Taylor-Maclaurin series
$$
\Brho(x) = \Brho(0) + \Brho'(0)x + \Brho''(0)\frac{x^2}{2}+O(x^3),
$$
then
$$
\begin{aligned}
r^2(x)-|\Brho'(0)x|^2 &= \left(\Brho'(0)x + \Brho''(0)\frac{x^2}{2}+O(x^2)\right)\cdot\left(\Brho'(0)x + \Brho''(0)\frac{x^2}{2}+O(x^2)\right)-|\Brho'(0)x|^2\\
&= (\Brho'(0)\cdot\Brho''(0))x^3 + O(x^4),
\end{aligned}
$$
therefore
$$
\frac{r^2(x)-|\Brho'(0)x|^2}{|\Brho'(0)x|^2} = \frac{\Brho'(0)\cdot\Brho''(0)}{|\Brho'(0)|^2}x + O(x^2)
$$
is smooth near $x=0$, which implies that \eqref{eq:derivation3p} is indeed smooth. Next, we use the decomposition
$$\log r(x) = \log \frac{r(x)}{|\Brho'(0)x|} + \log |\Brho'(0)| +  \log |x|,$$
to analyze the error of the punctured trapezoidal rule being applied to $s\cdot\tilde{\tau}\cdot\eta$, as follows
\begin{equation}
\begin{aligned}
&I[s\cdot\tilde{\tau}\cdot\eta] - T^0_h[s\cdot\tilde{\tau}\cdot\eta] \\
&= \int_{-a}^a \Big(-\log r(x)\Big)\,\tilde{\tau}(x)\eta(x)\,\d x - \sideset{}'\sum_{n=-M}^M \Big(-\log|nh|\Big)\,\tilde{\tau}(nh)\eta(nh)\,h\\
&= \left\{\int_{-a}^a \Bigg(-\log \frac{r(x)}{|\Brho'(0)x|}\Bigg)\tilde{\tau}(x)\eta(x)\,\d x - \sideset{}'\sum_{n=-M}^M \Bigg(-\log \frac{r(nh)}{|\Brho'(0)nh|}\Bigg)\tilde{\tau}(nh)\eta(nh)\,h\right\}\\
&\quad -\log|\Brho'(0)|\, \left\{\int_{-a}^a \tilde{\tau}(x)\eta(x)\,\d x - \sideset{}'\sum_{n=-M}^M \tilde{\tau}(nh)\eta(nh)\,h\right\}\\
&\quad + \left\{\int_{-a}^a -\log |x|\,\tilde{\tau}(x)\eta(x)\,\d x - \sideset{}'\sum_{n=-M}^M -\log |nh|\,\tilde{\tau}(nh)\eta(nh)\,h\right\}
\end{aligned}
\label{eq:derivation4}
\end{equation}
where, because \eqref{eq:derivation3p} is smooth, the terms in the first curly brackets of \eqref{eq:derivation4} happen to be the error of the regular trapezoidal rule applied to a smooth function (notice that the integrand is zero at $x=0$), which vanishes super-algebraically; the terms in the second curly brackets, analogous to \eqref{eq:derivation2.2}, converge to $-(h\,\log |\Brho'(0)|)\tilde{\tau}(0)$ super-algebraically; finally for the terms in the last curly brackets, one simply applies the quadrature \eqref{eq:quad_log|x|} of Theorem \ref{thm:quad_log|x|}, with $\tau$ replaced by $\tilde{\tau}$. Combining all these estimates, as well as \eqref{eq:smooth_trapz}, one concludes that \eqref{eq:derivation4} implies \eqref{eq:quad_on_curve_log|x|}.
\end{proof}

%sssssssssssssssssss
\subsection{Helmholtz kernels}
We now apply Theorem \ref{thm:quad_on_curve_log|x|} to construct formulae for the Helmholtz layer potentials. Consider the Helmholtz SLP $S_\kappa$ on a smooth closed curve $\Brho(x)$ evaluated at $\Brho(0)$,
\begin{equation}
S_\kappa[\tau](\Brho(0)) := \int_{-a}^a s_\kappa(r)\tau(x)\,|\Brho'(x)|\,\d x
\label{eq:helmholtz_slp_ker_curve}
\end{equation}
where $r\equiv r(x)=|\Brho(0)-\Brho(x)|$ and $\kappa\in\CC$ is the wavenumber, and where the kernel $s_\kappa$ has the form \cite[\S3.5]{colton2019inverse}
\begin{equation}
s_\kappa(r) := \frac{i}{4}H_0(\kappa r) = -\frac{1}{2\pi}\log(r)J_0(\kappa r) + \frac{c_\gamma}{2\pi} + \phi(r^2)
\label{eq:helmholtz_slp_ker}
\end{equation}
where $H_0$ and $J_0$ are, respectively, the Hankel and Bessel functions of the first kind of order $0$, where  $\phi(r^2)\equiv\phi(r(x)^2)$ is some smooth function of $x$ such that $\phi(0)=0$, and where $c_\gamma := \frac{\pi i}{2}-(\log\frac{\kappa}{2}+\gamma)$ such that $\gamma=0.5772\dots$ is Euler's constant. Analogous to Kress' analytic split \eqref{eq:kress_split}, we introduce the kernel split
$$
s_\kappa(r) = s(r)J_0(\kappa r)/(2\pi) + s_\kappa^{(1)}(r)
$$
where $s(r)=-\log r$, such that the component $s_\kappa^{(1)}(r):= s_\kappa(r)-s(r)J_0(\kappa r)/(2\pi)$ is smooth. Therefore we can split \eqref{eq:helmholtz_slp_ker_curve} as
\begin{equation}
    S_\kappa[\tau](\Brho(0)) \equiv I[s_\kappa\cdot \tilde{\tau}] = I[s\cdot\tilde{\tau}_{\kappa}^S] + I[s_\kappa^{(1)}\cdot\tilde{\tau}]
    \label{eq:helm_slp_split}
\end{equation}
where $\tilde{\tau}:=\tau\cdot|\Brho'|$ and 
\begin{equation}
\tilde{\tau}_{\kappa}^S(x):=J_0(\kappa\, r(x))\,\tau(x)\,|\Brho'(x)|/(2\pi)
\label{eq:derivation5}
\end{equation}
are smooth function. Notice that the singular integral $I[s\cdot\tilde{\tau}_{\kappa}^S]$ can be approximated by \eqref{eq:quad_on_curve_log|x|} with $\tilde\tau$ replaced by $\tilde{\tau}_{\kappa}^S$, giving
\begin{equation}
    I[s\cdot\tilde{\tau}_{\kappa}^S] = T^0_h[s\cdot\tilde{\tau}_\kappa^S] -\frac{h}{2\pi}\log(|\Brho'(0)|\,h)\tilde{\tau}(0)+h\sum_{j = 0}^{K}w_j\big(\tilde{\tau}_\kappa^S(jh)+\tilde{\tau}_\kappa^S(-jh)\big) + O(h^{2K+2})
    \label{eq:derivation6}
\end{equation}
where we have used the fact that $J_0(0)=1$. Then combining \eqref{eq:derivation6} with the PTR \eqref{eq:ptr} for $I[s_\kappa^{(1)}\cdot\tilde{\tau}]$ (which converges super-algebraically), we finally have
\begin{equation}
    S_\kappa[\tau](\Brho(0)) = T^0_h[s_\kappa\cdot\tilde{\tau}]+\frac{h}{2\pi}\Big(c_\gamma-\log(|\Brho'(0)|\,h)\Big)\tilde{\tau}(0)+h\sum_{j = 0}^{K}w_j\big(\tilde{\tau}_\kappa^S(jh)+\tilde{\tau}_\kappa^S(-jh)\big) + O(h^{2K+2})
\label{eq:quad_on_curve_H0}
\end{equation}
where, again, $\tilde{\tau}_\kappa^S$ is given by \eqref{eq:derivation5}.

Finally, we can also obtain the formulae for the Helmholtz double-layer potential (DLP), $D_\kappa$, and the normal derivative of the SLP, $D^*_\kappa$, using similar derivations. We will just state the formulae and omit the derivations. Using similar notations from \eqref{eq:helmholtz_slp_ker_curve} and $\tilde\tau\equiv\tau
\cdot|\Brho'|$, these layer potentials are given by
\begin{equation}
D_\kappa[\tau](\Brho(0)) := I[d_\kappa\cdot \tilde{\tau}]\quad\text{ and }\quad D^*_\kappa[\tau](\Brho(0)) := I[d^*_\kappa\cdot \tilde{\tau}]
\label{eq:helmholtz_dlp_ker_curve}
\end{equation}
where $d_\kappa(r) := \nn\cdot\nabla_{\boldsymbol{\rho}}s_\kappa(r)$ with $\nn\equiv\nn(x)$ being the unit outward normal at $\Brho(x)$, and where $d^*_\kappa = \nn_0\cdot\nabla_{\boldsymbol{\rho}_0}s_\kappa$ with $\nn_0:=\nn(0)$ and $\Brho_0:=\Brho(0)$. The corresponding corrected trapezoidal rules for $D_\kappa$ and $D^*_\kappa$ are, respectively,
\begin{align}
I[d_\kappa\cdot \tilde{\tau}] &= T^0_h[d_\kappa\cdot\tilde{\tau}]+h\,c_0\tilde{\tau}(0)+h\sum_{j = 1}^{K}w_j\big(\tilde{\tau}_\kappa^D(jh)+\tilde{\tau}_\kappa^D(-jh)\big) + O(h^{2K+2})
\label{eq:quad_on_curve_dnH0}\\
I[d^*_\kappa\cdot \tilde{\tau}] &= T^0_h[d^*_\kappa\cdot\tilde{\tau}]+h\,c_0\tilde{\tau}(0)+h\sum_{j = 1}^{K}w_j\big(\tilde{\tau}_\kappa^{D^*}(jh)+\tilde{\tau}_\kappa^{D^*}(-jh)\big) + O(h^{2K+2})
\label{eq:quad_on_curve_dn0H0}
\end{align}
where $c_0:=\frac{\Brho''(0)\cdot\nn(0)}{4\pi|\Brho'(0)|^2}$ the curvature at $\Brho_0$ scaled by $-\frac{1}{4\pi}$, and where
$$
\begin{aligned}
\tilde{\tau}_\kappa^D(x):=\kappa \, J_1(\kappa\,r(x))\frac{\rr(x)\cdot\nn(x)}{2\pi\,r(x)}\,\tilde{\tau}(x)\quad \text{ and }\quad
\tilde{\tau}_\kappa^{D^*}(x):=-\kappa \, J_1(\kappa\,r(x))\frac{\rr(x)\cdot\nn_0}{2\pi\,r(x)}\,\tilde{\tau}(x),
\end{aligned}
$$
with $J_1$ being the Bessel function of the first kind of order $1$. 

%sssssssssssssssssss
\section{Numerical experiments}
\label{sc:numerical}
In this section, we present numerical examples of solving BIEs associated with the Stokes and Helmholtz equations. In each case, we obtain a linear system of the form \eqref{eq:nystromlinsys}, where the matrix $\mathbf{K}$ is filled using a particular quadrature. Then the linear system is solved either directly by inverting the matrix or iteratively by GMRES. 

We compare our quadrature method with the three singular quadratures mentioned in the introduction: Kapur and Rokhlin's locally corrected trapezoidal quadrature \cite{kapur1997high}, Alpert's hybrid Gauss-trapezoidal quadrature \cite{alpert1999hybrid}, and Kress's spectral quadrature \cite[\S3.6]{colton2019inverse}. The correction weights for our quadrature are precomputed by solving the equations \eqref{eq:moment_eq_zeta_prime} and using the techniques described in Remark \ref{rmk:complex_step}. When implementing the Kapur-Rokhlin, Kress, and Alpert quadratures, we followed the survey \cite{hao2014high}.

\begin{figure}[htbp]
\centering
\begin{tabular}{ccc}
\textbf{(a)} & \textbf{(b)} &  \textbf{(c)} \\
\includegraphics[align=t,width=0.24\textwidth]{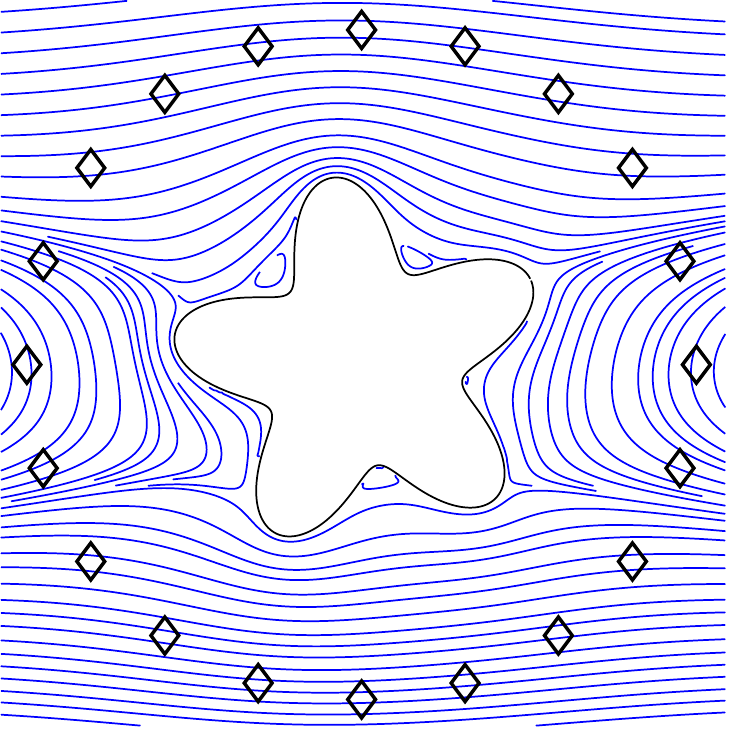} & \includegraphics[align=t,width=0.32\textwidth]{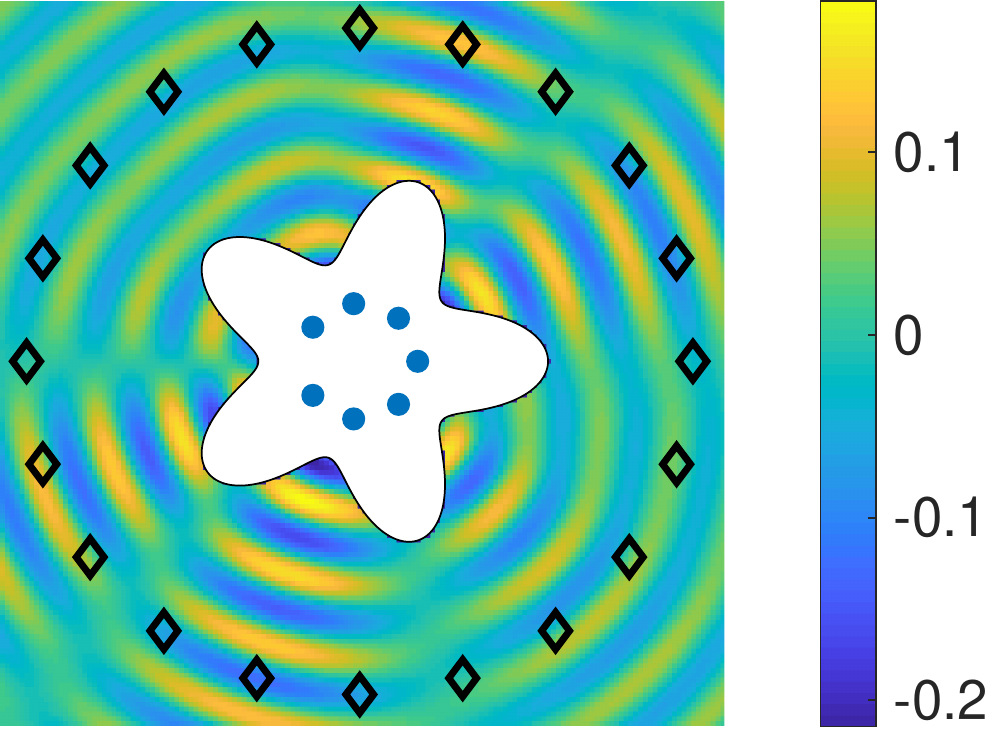} & \includegraphics[align=t,width=0.33\textwidth]{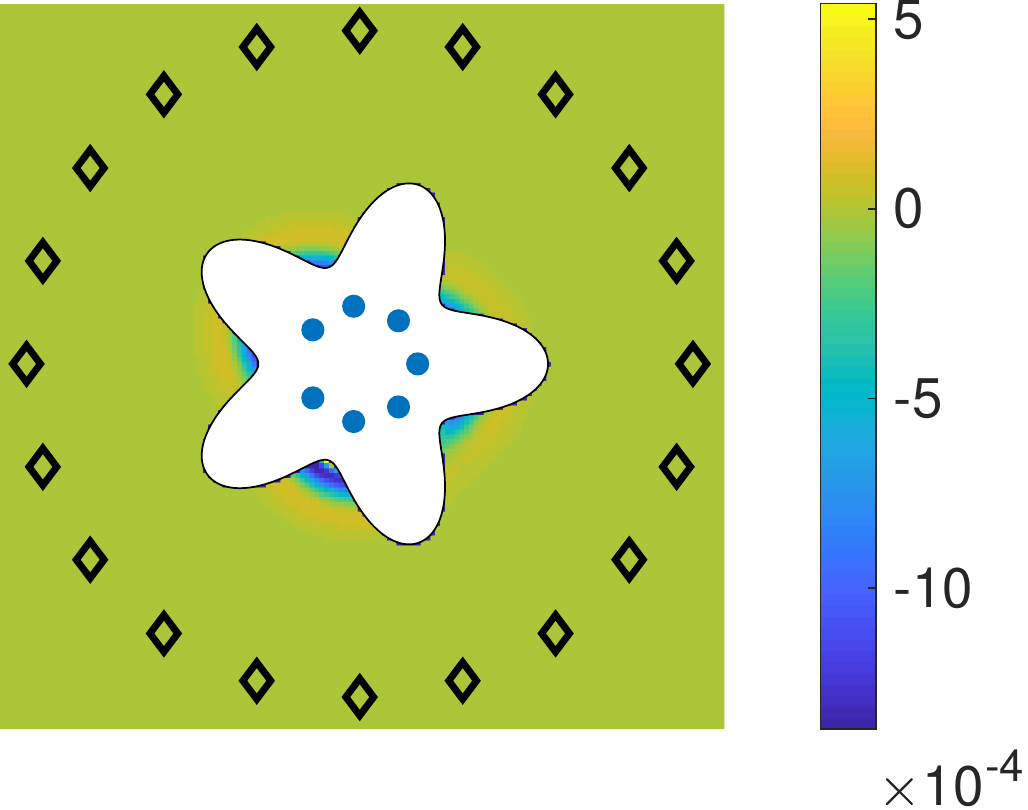}
\end{tabular}
\caption{Problem setup for the tests in Figures \ref{fig:quad_compare_stokes} and \ref{fig:quad_compare_helmholtz}. the Stokes problem \eqref{eq:stokes_bie} and the Helmholtz problem \eqref{eq:helmholtz_diri}. 
In all cases, the star-shaped geometry is parameterized by the polar function $p(\theta) = 1 + 0.3\cos(5\theta)$, while the diamonds represent testing locations.
(a) Streamlines of a shear flow $\uu^\infty(x_1,x_2)=(5x_2,0)$ around an island with no-slip boundary condition. (b) Real part of a wave field generated at the source locations indicated by the dots. The wavenumber is $\kappa=12.5$.
(c) Same as in (b), except that the wavenumber is now $\kappa=12.5+10i$, so the wave decays exponentially}\label{fig:problem_setup}
\end{figure}

%ppppppppppppp
\paragraph{Stokes problem.} As shown in Figure \ref{fig:problem_setup}(a), consider a viscous shear flow $\uu^\infty(x_1,x_2)=(5x_2,0)$ around an island whose boundary is a smooth closed curve $\Gamma$, with no-slip boundary conditions on $\Gamma$. Let $\uu$ be the true velocity field and $p$ its associated pressure field, then $(\uu,p)$ is described by the exterior Dirichlet problem for the Stokes equation \cite[\S2.3.2]{hsiao2008boundary}
$$
-\Delta \uu + \nabla p = 0 \;\text{ and }\; \nabla\cdot\uu = 0  \;\text{ in }\Omega,\quad \uu = \mathbf{0} \;\text{ on }\Gamma,\quad
\uu\to\uu^\infty\;\text{ as }|\xx|\to\infty$$
The integral equation formulation for this problem is obtained using the mixed potential representation $\uu(\xx) = \uu^\infty(\xx) +(\mathbf{S} + \mathbf{D})[\Btau](\xx), \xx\in\Omega$ for the velocity \cite[\S4.7]{pozrikidis1992boundary}, where the integral operators
$$
\begin{aligned}
\mathbf{S}[\Btau](\xx)&:=\frac{1}{4\pi}\int_\Gamma\left(-\log r\,\mathbf{I} + \frac{\rr\otimes\rr}{r^2}\right)\Btau(\Brho )\,\d s(\Brho ),\\
\mathbf{D}[\Btau](\xx)&:=\frac{1}{\pi}\int_\Gamma\left(\frac{\rr\cdot\nn(\Brho )}{r^2}\frac{\rr\otimes\rr}{r^2}\right)\Btau(\Brho )\,\d s(\Brho )
\end{aligned}
$$
are the Stokes SLP and DLP in 2D, where $\rr=\xx-\Brho , r=|\rr|$ and $\nn(\Brho)$ is the unit outward normal to $\Gamma$ at $\Brho $, and where $\otimes$ denotes the tensor product. Then the vector-valued unknown density function $\Btau$ is the solution of the following BIE
\begin{equation}
\left(\mathbf{S} + \mathbf{D}+\frac{1}{2}\right)[\Btau](\xx) = -\uu^\infty(\xx),\quad \xx\in\Gamma.
\label{eq:stokes_bie}
\end{equation}
As $\Brho \to\xx\in\Gamma$, the only singular component in the linear operator is the $\log r$ term in $\mathbf{S}$, which can be efficiently handled by the corrected trapezoidal rule \eqref{eq:quad_on_curve_log|x|}.

Figure \ref{fig:quad_compare_stokes} compares the convergence results of solving the Stokes problem using different quadrature methods. We see that all the quadratures have the expected convergence rates, with the Kapur-Rokhlin quadrature yielding higher absolute errors due to its larger correction weights. The Kress quadrature is the most accurate at virtually any given $N$, but the new quadrature of order 16 is remarkably close.

\begin{figure}
    \centering
    \includegraphics[width=0.70\textwidth]{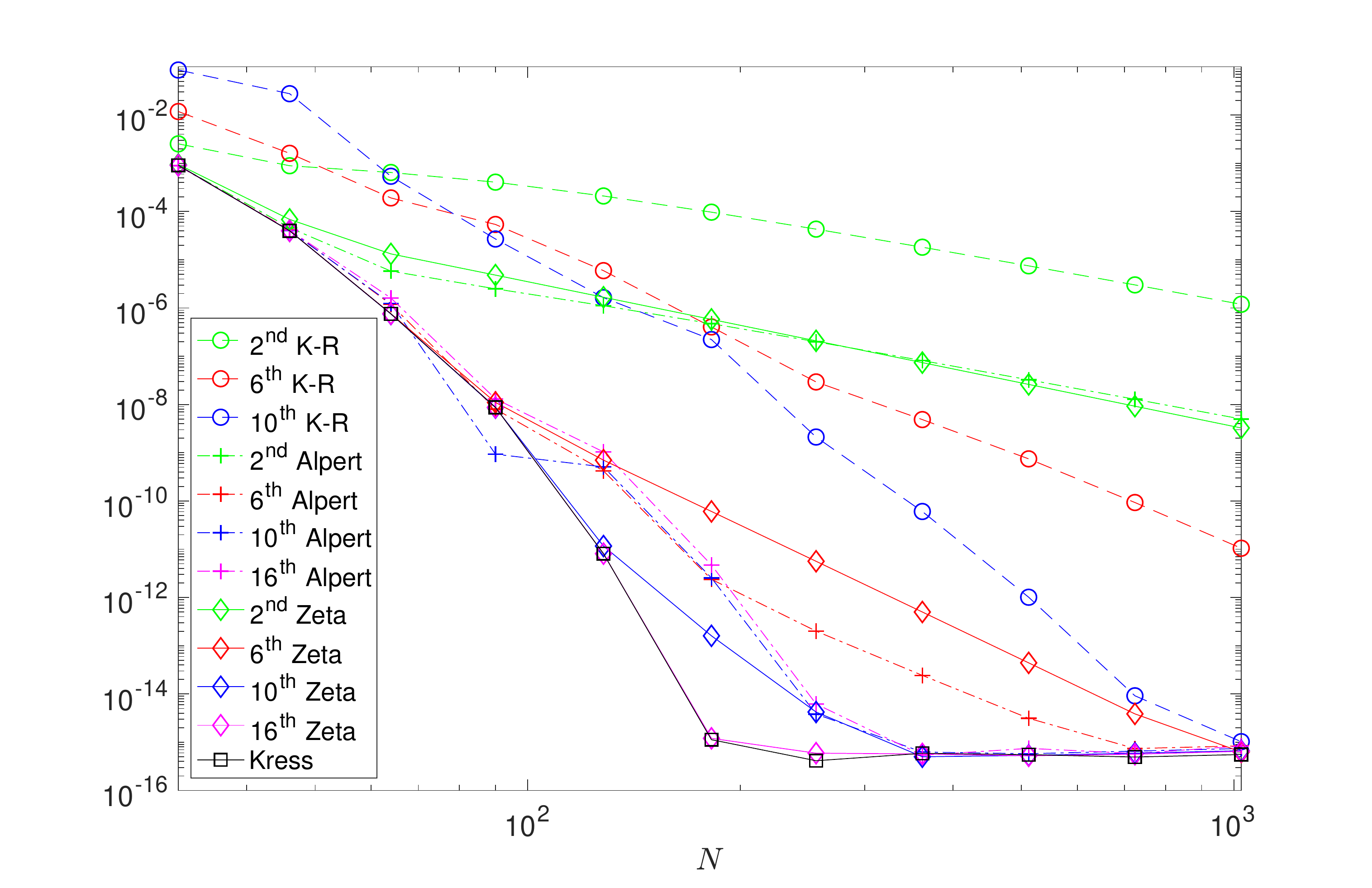}
    \caption{Comparison of singular quadratures for solving the Stokes problem \eqref{eq:stokes_bie} using the setup in Figure \ref{fig:problem_setup}(a). The reference solution is obtained using the Kress quadrature with $N=2000$ points on $\Gamma$.}
    \label{fig:quad_compare_stokes}
\end{figure}

%ppppppppppppp
\paragraph{Helmholtz problem.} As shown in Figure \ref{fig:problem_setup}(b) or (c), consider the Helmholtz Dirichlet problem exterior to the curve $\Gamma$, with boundary data $f$ and $u$ satisfying the Sommerfeld radiation condition,
$$
-\Delta u - \kappa^2 u = 0 \;\text{ in }\Omega,\quad u = f \;\text{ on }\Gamma,\quad \lim_{|\xx|\to\infty}|\xx|^{1/2}\left(\frac{\partial u}{\partial|\xx|}-i\,\kappa u\right)=0
$$
where $\kappa\in\CC$ is the wavenumber. For integral equation reformulation, consider the mixed potential assumption $u(\xx) = (D_\kappa - i\,\kappa S_\kappa)[\tau](\xx),\xx\in\Omega$ (see \cite[\S3]{colton2019inverse},\cite{bremer2015high}), where $S_\kappa$ and $D_\kappa$ are the Helmholtz SLP and DLP as defined in the previous section, then the unknown density function $\tau$ is the solution of the BIE
\begin{equation}
\left(\frac{1}{2} + D_\kappa - i\,\kappa S_\kappa\right)[\tau](\xx) = f(\xx),\quad \xx\in\Gamma \label{eq:helmholtz_diri}
\end{equation}
The integral operators in this system can be discretized using the quadratures \eqref{eq:quad_on_curve_H0} and \eqref{eq:quad_on_curve_dnH0}.

\begin{figure}
\small
    \centering
    \begin{tabular}{llp{0.73\textwidth}}
         \textbf{(a)} & 
         \begin{tabular}{l}
              $5\lambda$ \\
              ($\kappa=12.5$)
         \end{tabular} & \includegraphics[align=c,width=0.64\textwidth]{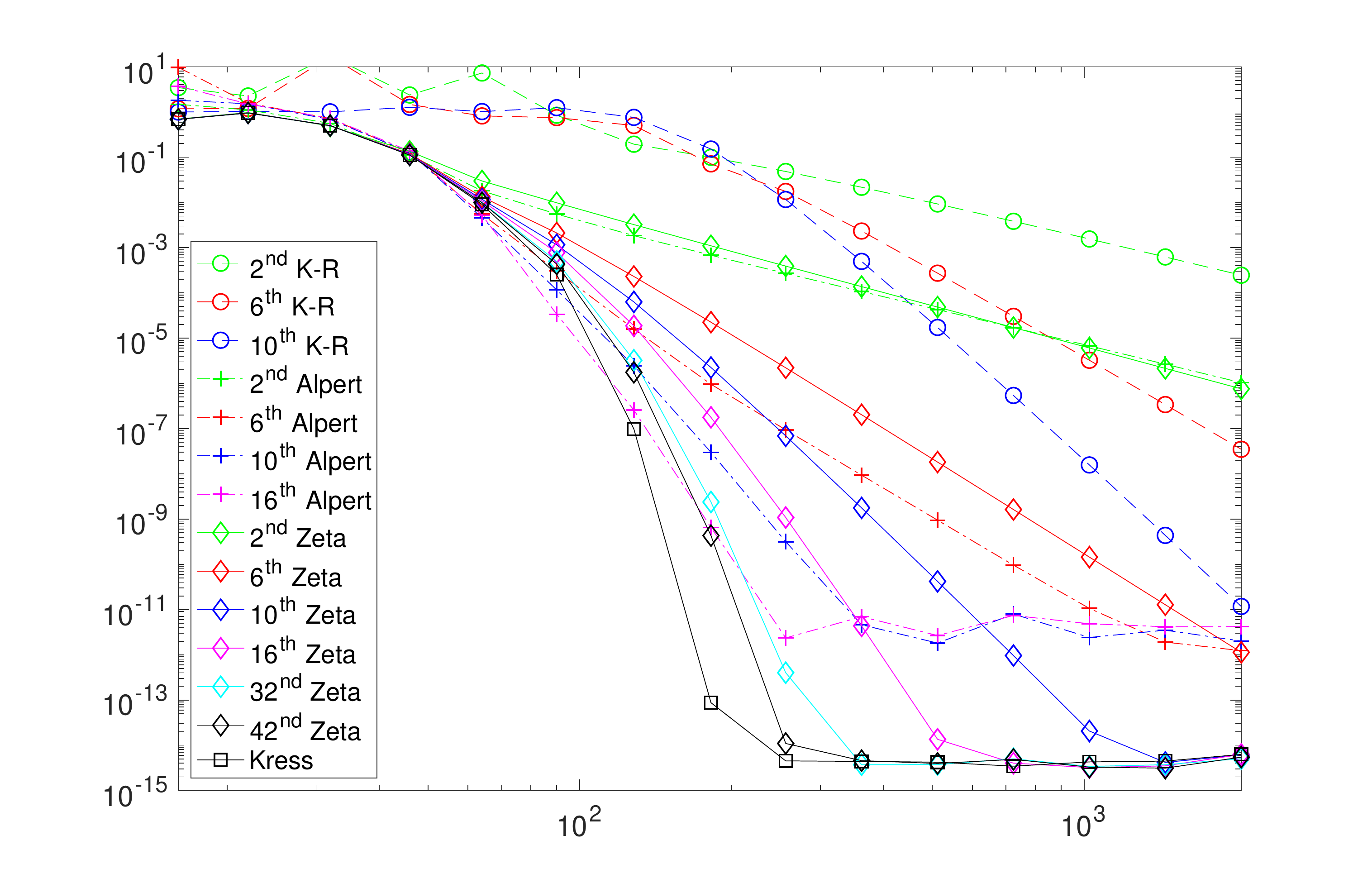} \\
         \textbf{(b)} & 
         \begin{tabular}{l}
              $50\lambda$ \\
              ($\kappa=125$)
         \end{tabular} & \includegraphics[align=c,width=0.64\textwidth]{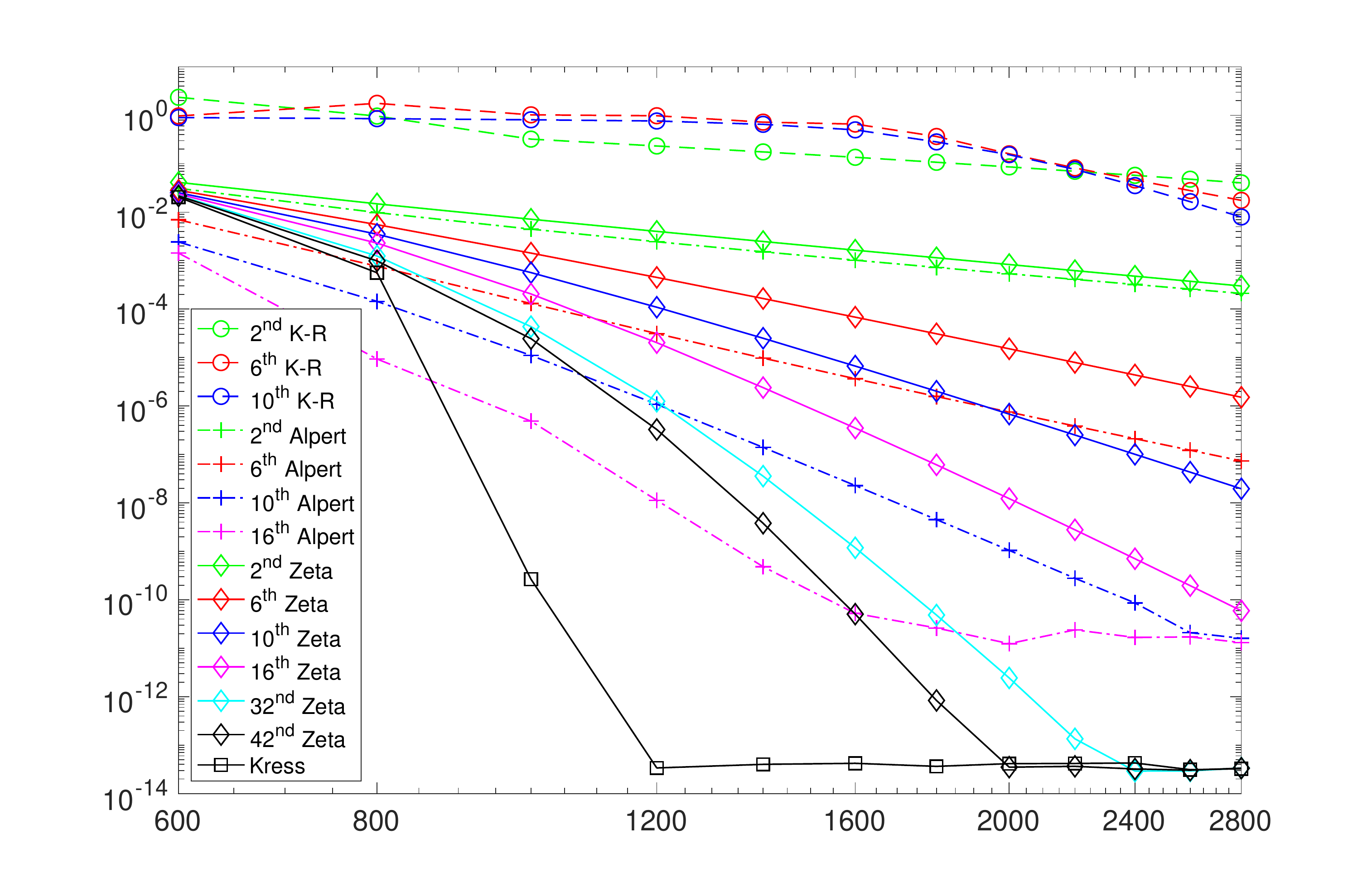} \\
         \textbf{(c)} & 
         \begin{tabular}{l}
              $5\lambda$\\ exp. decay \\
              ($\kappa=12.5+10i$)
         \end{tabular} & \includegraphics[align=c,width=0.64\textwidth]{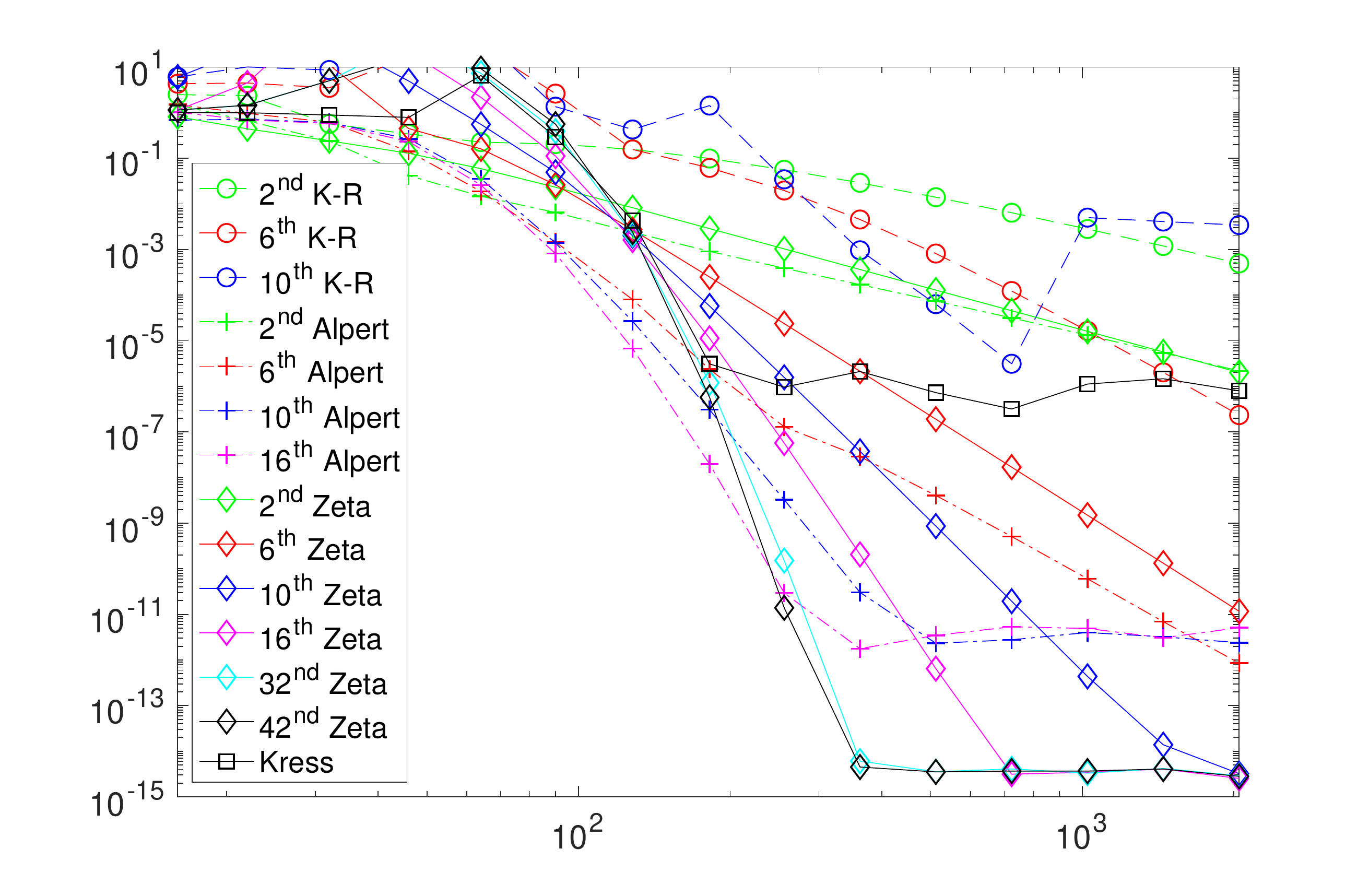}
    \end{tabular}
    \caption{Comparison of singular quadratures for solving the Helmholtz problem \eqref{eq:helmholtz_diri}. Horizontal axes are $N$ (number of points), vertical axes are relative errors. Tests in (a) and (c) correspond to the problem setups in Figure \ref{fig:problem_setup}(b) and \ref{fig:problem_setup}(c), respectively.}
    \label{fig:quad_compare_helmholtz}
\end{figure}

Figure \ref{fig:quad_compare_helmholtz} compares the convergence results of solving this Helmholtz problem using different quadrature methods. We look at the results for different values of $\kappa$ that corresponds to $5$ or $50$ wavelengths across the geometry's diameter.  We observe the following:
\begin{itemize}
\item \textbf{Accuracy comparison.} For a given number of unknowns, the Alpert quadrature achieves slightly higher accuracy than our quadrature of the same order of correction, and the difference is larger for a larger $\kappa$ (higher frequency). The Kapur-Rokhlin quadrature has a harder time to converge to a given accuracy as the frequency becomes higher. The Kress quadrature consistently yields best accuracy for any real $\kappa$. The saturation error of the Alpert quadrature is higher than the other quadratures due to interpolation.
\item \textbf{Work comparison.} For a given order, the Alpert correction requires more work then the other locally corrected quadratures because it requires interpolation for each non-uniform grid point from the nearby uniform grid, leading to modification of a larger bandwidth of the $\mathbf{K}$ matrix \eqref{eq:niceK}. Consequently the bandwidth modified by the $16^\text{th}$ order Alpert quadrature, for example, is almost as large as what is modified by our $42^\text{nd}$ order quadrature. So with the same amount of work, our quadrature is able to obtain a higher accuracy. The Kapur-Rokhlin correction always costs the least amount of work for a given order, since it is an on-grid correction whose weights are independent of the specific form of the kernel. 
\item \textbf{Effect of fast decaying waves.} When $\mathrm{Im}\,\kappa>0$ which corresponds to exponentially decaying waves, the convergence of the Kress quadrature stagnates at a lower accuracy (see Fig. \ref{fig:quad_compare_helmholtz}(c)). This behavior can be explained as follows. In the Kress kernel split
$$\pi i\,H_0(\kappa\,r(x)) = -J_0(\kappa\,r(x))\log\left(4\sin^2\frac{x}{2}\right)+\psi(x),$$
when $r(x)$ becomes bigger, the left-hand side decays exponentially while the Bessel function $J_0$ on the right-hand side grows exponentially, leading to errors due to numerical cancellation. Our local correction method, on the other hand, is immune to such effect since locally $r(x)$ never grown big. In Figure \ref{fig:quad_compare_helmholtz}(c), we let $\mathrm{Im}\,\kappa=10$ to produce a more dramatic stagnation, but it starts to manifest when $\mathrm{Im}\,\kappa=5$.
\end{itemize}

Finally, Table \ref{tab:cond} compares the condition number of the matrix resulted from discretizing \eqref{eq:helmholtz_diri} with different quadratures. As expected, the Kapur-Rokhlin quadrature gives rise to condition numbers that are many times larger than the other quadratures due to its large correction weights. All other quadratures give rise to matrices that are well-conditioned and, accordingly, the number of GMRES iterations required is constant for any sufficiently large $N$ --- except for the case of fast decaying waves mentioned above, where the loss of digits of the Kress quadrature corresponds to a condition number $10^6$ which is orders of magnitude larger than other quadratures.

\begin{table}[htbp]
\small
    \centering
    \begin{tabular}{r*{10}{c}}
    \hline
        Scheme & 6th & 10th & 6th & 10th & 16th & 6th & 10th & 16th & 42th & Kress  \\
        & K-R & K-R & Alpert & Alpert & Alpert & Zeta & Zeta & Zeta & Zeta & \\
        \hline
        ($\mathrm{Im}\,\kappa=0$) cond \# & 30.5 & 182 & 5.32 & 5.38 & 5.34 & 5.32 & 5.32 & 5.32 & 5.32 & 5.32 \\
         \# iters & 139 & 525 & 34 & 34 & 34 & 34 & 34 & 34 & 34 & 34\\
        \hline
        ($\mathrm{Im}\,\kappa=10$) cond \# & 9.69 & 94.0 & 1.80 & 1.82 & 1.81 & 1.80 & 1.80 & 1.80 & 1.80 & $10^6$ \\
        \# iters & 60 & 470 & 18 & 18 & 18 & 18 & 18 & 18 & 18 & 45
    \end{tabular}
    \caption{Condition numbers of the matrix relsulting from discretizing \eqref{eq:helmholtz_diri} with different quadratures and the numbers of GMRES iterations to reach the residual error $10^{-14}$. The top data correspond to $\kappa=12.5$ (Fig. \ref{fig:quad_compare_helmholtz}(a)) and the bottom data to $\kappa=12.5+10i$ (Fig. \ref{fig:quad_compare_helmholtz}(c)).}
    \label{tab:cond}
\end{table}

\section{Conclusion}

In this paper, we described a technique for modifying the trapezoidal quadrature rule to attain high order convergence for integral operators with weakly singular kernels.
The new ``zeta-corrected'' quadrature rule builds correction weights by fitting the error moments on a local sub-grid near the singular point. In the case that the singularity is an algebraic or logarithmic branch point, we have shown that the corresponding error expansion has coefficients expressible as Riemann zeta function values or their derivatives, hence the name ``zeta correction.'' Since the correction is local, our quadrature can be combined with boundary error corrections to also handle non-periodic intervals and open curves.

The zeta correction technique can naturally be generalized to higher dimensions. For instance, using the Epstein zeta function \cite{epstein1903theorie} (a generalization of the  Riemann zeta function), correction weights for the discretization of a BIE on a surface in three dimensions are derived in \cite{wu2020corrected}.

BIE solvers based on the zeta-corrected quadrature are fast, accurate, stable, and easy to implement. The code accompanying this manuscript is published on GitHub which can be accessed at the following link address.
\begin{center}
    \url{https://github.com/bobbielf2/ZetaTrap2D}
\end{center}

\paragraph{Acknowledgements}
The authors would like to thank Alex Barnett for sharing valuable perspectives and insights.
The work reported was supported by the Office of Naval Research
(grant N00014-18-1-2354), and by the National Science Foundation (grants DMS-1620472 and DMS-2012606).

\bibliography{bobi_quadr} 

\begin{thebibliography}{10}
\providecommand{\url}[1]{{#1}}
\providecommand{\urlprefix}{URL }
\expandafter\ifx\csname urlstyle\endcsname\relax
  \providecommand{\doi}[1]{DOI~\discretionary{}{}{}#1}\else
  \providecommand{\doi}{DOI~\discretionary{}{}{}\begingroup
  \urlstyle{rm}\Url}\fi

\bibitem{aguilar2002high}
Aguilar, J.C., Chen, Y.: High-order corrected trapezoidal quadrature rules for
  functions with a logarithmic singularity in 2-{D}.
\newblock Computers \& Mathematics with Applications \textbf{44}(8-9),
  1031--1039 (2002)

\bibitem{alpert1995high}
Alpert, B.K.: High-order quadratures for integral operators with singular
  kernels.
\newblock Journal of computational and applied mathematics \textbf{60}(3),
  367--378 (1995)

\bibitem{alpert1999hybrid}
Alpert, B.K.: Hybrid {G}auss-trapezoidal quadrature rules.
\newblock SIAM Journal on Scientific Computing \textbf{20}(5), 1551--1584
  (1999)

\bibitem{borwein2013lattice}
Borwein, J.M., Glasser, M., McPhedran, R., Wan, J., Zucker, I.: Lattice sums
  then and now.
\newblock 150. Cambridge University Press (2013)

\bibitem{bremer2015high}
Bremer, J., Gillman, A., Martinsson, P.G.: A high-order accurate accelerated
  direct solver for acoustic scattering from surfaces.
\newblock BIT Numerical Mathematics \textbf{55}(2), 367--397 (2015)

\bibitem{colton2019inverse}
Colton, D., Kress, R.: Inverse acoustic and electromagnetic scattering theory,
  vol.~93.
\newblock Springer Nature (2019)

\bibitem{epstein1903theorie}
Epstein, P.: Zur theorie allgemeiner zetafunctionen.
\newblock Mathematische Annalen \textbf{56}(4), 615--644 (1903)

\bibitem{rokhlin1987}
Greengard, L., Rokhlin, V.: A fast algorithm for particle simulations.
\newblock J. Comput. Phys. \textbf{73}(2), 325--348 (1987)

\bibitem{hao2014high}
Hao, S., Barnett, A.H., Martinsson, P.G., Young, P.: High-order accurate
  methods for {N}ystr{\"o}m discretization of integral equations on smooth
  curves in the plane.
\newblock Advances in Computational Mathematics \textbf{40}(1), 245--272 (2014)

\bibitem{hsiao2008boundary}
Hsiao, G.C., Wendland, W.L.: Boundary integral equations.
\newblock Springer (2008)

\bibitem{kapur1997high}
Kapur, S., Rokhlin, V.: High-order corrected trapezoidal quadrature rules for
  singular functions.
\newblock SIAM Journal on Numerical Analysis \textbf{34}(4), 1331--1356 (1997)

\bibitem{keast1979structure}
Keast, P., Lyness, J.N.: On the structure of fully symmetric multidimensional
  quadrature rules.
\newblock SIAM Journal on Numerical Analysis \textbf{16}(1), 11--29 (1979)

\bibitem{kress1991boundary}
Kress, R.: Boundary integral equations in time-harmonic acoustic scattering.
\newblock Mathematical and Computer Modelling \textbf{15}(3-5), 229--243 (1991)

\bibitem{kress2014linear}
Kress, R.: Linear Integral Equations, \emph{Applied Mathematical Sciences},
  vol.~82, 3 edn.
\newblock Springer-Verlag New York (2014)

\bibitem{marin2014corrected}
Marin, O., Runborg, O., Tornberg, A.K.: Corrected trapezoidal rules for a class
  of singular functions.
\newblock IMA Journal of Numerical Analysis \textbf{34}(4), 1509--1540 (2014)

\bibitem{2019_martinsson_book}
Martinsson, P.G.: Fast Direct Solvers for Elliptic PDEs, \emph{CBMS-NSF
  conference series}, vol. CB96.
\newblock SIAM (2019)

\bibitem{navot1961extension}
Navot, I.: An extension of the {E}uler-{M}aclaurin summation formula to
  functions with a branch singularity.
\newblock Journal of Mathematics and Physics \textbf{40}(1-4), 271--276 (1961)

\bibitem{navot1962further}
Navot, I.: A further extension of the {E}uler-{M}aclaurin summation formula.
\newblock Journal of Mathematics and Physics \textbf{41}(1-4), 155--163 (1962)

\bibitem{pozrikidis1992boundary}
Pozrikidis, C., et~al.: Boundary integral and singularity methods for
  linearized viscous flow.
\newblock Cambridge University Press (1992)

\bibitem{sidi1988quadrature}
Sidi, A., Israeli, M.: Quadrature methods for periodic singular and weakly
  singular {F}redholm integral equations.
\newblock Journal of Scientific Computing \textbf{3}(2), 201--231 (1988)

\bibitem{squire1998using}
Squire, W., Trapp, G.: Using complex variables to estimate derivatives of real
  functions.
\newblock SIAM review \textbf{40}(1), 110--112 (1998)

\bibitem{trefethen2014exponentially}
Trefethen, L.N., Weideman, J.: The exponentially convergent trapezoidal rule.
\newblock SIAM Review \textbf{56}(3), 385--458 (2014)

\bibitem{wu2020corrected}
Wu, B., Martinsson, P.G.: Corrected trapezoidal rules for boundary integral
  equations in three dimensions.
\newblock arXiv: \href{https://arxiv.org/abs/2007.02512}{2007.02512} (2020)

\bibitem{wu2020zeta}
Wu, B., Martinsson, P.G.: Zeta correction: A new approach to constructing
  corrected trapezoidal quadrature rules for singular integral operators.
\newblock arXiv: \href{https://arxiv.org/abs/2007.13898}{2007.13898} (2020)

\end{thebibliography}
\bibliographystyle{abbrv}

\end{document}